\documentclass[12pt, oneside,reqno]{amsart}
\usepackage{amsaddr} 
\usepackage[utf8]{inputenc}
\usepackage[german,english]{babel}      
\usepackage{amsmath}
\usepackage{amsfonts}
\usepackage{amssymb}
\usepackage{amsthm}
\usepackage{mathtools}
\usepackage{mathrsfs}                   
\usepackage{appendix}                   
\usepackage{tikz}                       
\usetikzlibrary{arrows}
\usepackage{wrapfig}

\newcommand{\lra}{\longrightarrow}

\newcommand{\C}{{\mathbb{C}}}

\newcommand{\HH}{{\mathbb{H}}}         

\newcommand{\R}{{\mathbb{R}}}

\newcommand{\Z}{{\mathbb{Z}}}          

\newcommand{\Dd}{{\mathcal{D}}}        

\newcommand{\Hh}{{\mathcal{H}}}        

\newcommand{\Ss}{{\mathcal{S}}}        

\newcommand{\Zz}{{\mathcal{Z}}}

\newcommand{\hk}{\mathsf{H}}

\newcommand{\swann}{\mathcal{U}(N)}    
\newcommand{\euler}{\mathcal{X}_{0}}   
\newcommand{\imag}{\mathrm{\mathbf{i}}}

\newcommand{\sst}{\scriptscriptstyle}
\newcommand{\mc}[1]{\mathcal{#1}}

\newcommand{\ms}[1]{\mathsf{#1}}
\newcommand{\scr}[1]{\mathscr{#1}}
\newcommand{\mf}[1]{\mathfrak{#1}}
\newcommand{\mbb}[1]{\mathbb{#1}}

\newcommand{\norm}[1]{\left\| #1 \right\|}
\newcommand{\abs}[1]{\left\lvert #1 \right\rvert}
\newcommand{\pair}[1]{\left\langle #1 \right\rangle}

\newcommand{\eqst}[1]{\begin{equation*} #1 
                      \end{equation*}}
\newcommand{\eq}[1]{\begin{equation} #1
                    \end{equation}}
\newcommand{\alst}[1]{\begin{align*} #1
                      \end{align*}}
\newcommand{\al}[1]{\begin{align} #1
                      \end{align}}
                      
\theoremstyle{plain}
\newtheorem{thm}{Theorem}[section]
\newtheorem{lem}[thm]{Lemma}
\newtheorem*{lemma*}{Lemma}
\newtheorem{prop}[thm]{Proposition}

\theoremstyle{definition}

\theoremstyle{definition}
\newtheorem{defn}{Definition}

\newtheorem{rmk}{Remark}

\theoremstyle{remark}
\newtheorem*{question*}{Question}

\DeclareMathOperator{\id}{id}
\DeclareMathOperator{\End}{End}
\DeclareMathOperator{\aut}{Aut}
\DeclareMathOperator{\Map}{Map}
\DeclareMathOperator{\Hom}{Hom}
\DeclareMathOperator{\pr}{proj}

\DeclareMathOperator{\im}{Im}

\DeclareMathOperator{\grad}{grad}

\usepackage{hyperref}
\hypersetup{
    colorlinks = true,%
    linkcolor = blue,%
    citecolor = blue,%
    pdfcreator={LaTeX},
    pdftitle={generalised Seiberg-Witten equations and almost-Hermitian geometry },%
    pdfauthor={Varun Thakre},%
    pdfkeywords={gauge theory, generalised dirac operator, hyperkahler manifolds, symplectic structures, Seiberg-Witten, almost-complex geometry},%
}

\allowdisplaybreaks

\begin{document}

\title[Generalised Seiberg-Witten equations and almost-Hermitian geometry]{Generalised Seiberg-Witten equations and almost-Hermitian geometry}

\author[V. Thakre]{Varun Thakre}

\address{International Centre for Theoretical Sciences (ICTS-TIFR), Hesaraghatta, Hobli, Bengaluru 560089, India}

\email{varun.thakre@icts.res.in}


\date{Revised on \today }


\begin{abstract}

In this article, we study a generalisation of the Seiberg-Witten equations, replacing the spinor representation with a hyperK\"ahler manifold equipped with certain symmetries. Central to this is the construction of a (non-linear) Dirac operator acting on the sections of the non-linear fibre-bundle. For hyperK\"ahler manifolds admitting a hyperK\"ahler potential, we derive a transformation formula for the Dirac operator under the conformal change of metric on the base manifold.

As an application, we show that when the hyperK\"ahler manifold is of dimension four, then away from a singular set, the equations can be expressed as a second order PDE in terms of almost-complex structure on the base manifold and a conformal factor. This extends a result of Donaldson to generalised Seiberg-Witten equations.

\end{abstract}

\maketitle

\section{Introduction}

Let $X$ be a 4-dimensional, oriented, smooth, Riemannian manifold and let $Q \rightarrow X$ be a ${\rm Spin}$-structure. A spinor bundle over $X$ is a vector bundle associated to $Q$, with typical fibre $\HH$. The idea for generalisation is to replace the spinor representation with a hyperK\"ahler manifold $(M, g_{\sst M}, I_1, I_2, I_3)$ equipped with an isometric action of ${\rm Sp}(1)$ (or ${\rm SO}(3)$) which \emph{permutes} the complex structures on $M$. We will often refer to $M$ as the \emph{target hyperK\"ahler manifold}. The sections of the non-linear fibre-bundle now play the role of spinors. The interplay between the ${\rm Sp}(1)$ (or ${\rm SO}(3)$) action and the quaternionic structure on $M$ allows one to define the Clifford multiplication. Composing the Clifford multiplication with the covariant derivative gives the generalised Dirac operator, which we denote by $\Dd$.

In order to define a generalisation of the Seiberg-Witten equations, we need additionally a twisting principal $G$-bundle $P_G \rightarrow X$, with a tri-Hamiltonian action of $G$ on $M$. The action gives rise to a hyperK\"ahler moment map $\mu: M \longrightarrow \mf{sp}(1)^{\ast} \otimes \mf{g}^{\ast}$. For a connection $A$ on $P_G$ and a spinor $u$ the 4-dimensional generalised Seiberg-Witten equations on $X$ are the following system of equations
\eq{
\label{eq: intro gen. sw on X}
\left\{\begin{array}{l}
\Dd_{A} u = 0 \\
F^{+}_{A} - \mu \circ u = 0
\end{array}\right.
}

where $\Dd_{A}$ is a twisted Dirac operator for a connection $A$ on $P_G$.

This non-linear generalisation of the Dirac operator is well-known to physicists and has been used in the study of gauged, non-linear $\sigma$-models \cite{anselmi-fre}. The 3-dimensional version of equations \eqref{eq: intro gen. sw on X} was studied by Taubes \cite{taubes} (see also \cite{martin}). The 4-dimensional generalisation was considered by Pidstrygach \cite{victor}, Schumacher \cite{henrik} and Haydys \cite{haydys}. The moduli spaces of solutions to \eqref{eq: intro gen. sw on X} makes for an interesting study, especially because of its application to gauge theories on manifolds with special holonomies (cf. \cite{haydys15}, \cite{witten11}). Many well-known gauge-theoretic equations like the ${\rm PU}(2)$-monopole equations \cite{feehan-leness1998}, the Vafa-Witten equations \cite{vafa-witten1994}, ${\rm Pin}(2)$-monopole equations \cite{manolescu2016}, the non-Abelian monopole equations \cite{teleman00}, etc. can be treated as special cases of this generalisation.

It is possible to obtain the target hyperK\"ahler manifold with requisite symmetries from \emph{Swann's construction} \cite{swann}, \cite{boyergalickimann2}. Starting with a quaternionic K\"ahler manifold $N$ of positive scalar curvature, Swann constructs a fibration $\swann \rightarrow N$, whose total space admits a hyperK\"ahler structure. Such manifolds are characterised by the existence of a hyperK\"ahler potential. Alternatively, the permuting ${\rm Sp}(1)$-action extends to a homothetic action of $\HH^{\ast}$. The bundle construction commutes with the hyperK\"ahler quotient construction of Hitchin, Karlhede, Lindstr\"om and Ro\v cek \cite{hklr} and the quaternionic K\"ahler quotient construction of Galicki and Lawson \cite{galilaw}. As a result, many examples of (finite dimensional) hyperK\"ahler manifolds with homothetic $\HH^{\ast}$-action can be obtained via hyperK\"ahler reduction of $\HH^{n}$.

With $M = \swann$, we derive a transformation formula for the generalised Dirac operator, under the conformal change of metric on the base manifold. Since $\swann$ admits a natural homothetic action of $\R^+$, this setting allows one to make sense of ``weighted spinors".

Let $\pi_1: P_{{\rm CO}(4)} \rightarrow X$ be the bundle of conformal frames with respect to the conformal class $[g_{\sst X}]$ and $P_G \rightarrow X$ be a principal $G$-bundle over $X$. Assume that the action of $G$ on $M$ is tri-Hamiltonian. Let $\widetilde{\pi}: \widetilde{Q} \rightarrow X$ denote the conformal ${\rm Spin}^G(4)$-bundle, which is a double cover of $P_{{\rm CO}(4)} \times_X P_G$.

\begin{thm} \label{thm: main thm 1}
Let $f$ be a smooth, real-valued function on $X$ and let $u$ be a (generalised) spinor. Consider the metric $g'_{\sst X} := e^{2f}g_{\sst X}$ in the conformal class $[g_{\sst X}]$ and let $\varphi'$ and $\varphi$ be the Levi-Civita connections associated to $g_{\sst X}$ and $g'_{\sst X}$ respectively. For a fixed connection $A$ on $P_G$, denote by $A_{\varphi}$ and $A_{\varphi'}$ the corresponding lifts to $\widetilde{Q}$. Then, the associated generalised Dirac operators $\Dd_{A_{\varphi}}$ and $\Dd_{A_{\varphi'}}$ are related as
\eq{
\Dd_{A_{\varphi'}} (\scr{B}u) = \scr{B} \Big( de^{-5/2\pi_1^{\ast}f}\Dd_{A_{\varphi}}(e^{3/2\pi_1^{\ast}f}u) \Big)
}
where, $\scr{B}$ is the lift of the automorphism $ B: P_{{\rm CO}(4)} \longrightarrow P_{{\rm CO}(4)}$, given by $p \longmapsto e^{-f} p,$ and $de^{-5/2\pi_1^{\ast}f}$ is the action of $e^{-5/2\pi_1^{\ast}f}$ by differential on $TM$.
\end{thm}
For $M=\HH$, the result was proved by Hitchin \cite{hitchin1974}.

Assume that $M = \swann$ is a 4-dimensional hyperK\"ahler manifold. Using the above theorem, we show that away from a singular set, the generalised Seiberg-Witten equations can be interpreted in terms of almost-complex geometry of the underlying 4-manifold, as equations for a compatible almost-complex structure and a real-valued function which is associated to a conformal factor. Recall that on a Riemannian 4-manifold $(X, g_{\sst X})$, the compatible almost-complex structures on $X$ are parametrized by sections of the twistor bundle $\Zz$, which is a sphere bundle in $\Lambda^+$. Thus the almost-complex structures can be thought of as self-dual, 2-forms $\Omega$ with $\abs{\Omega} = 1$. An almost-complex structure gives a splitting of $\Lambda^+$ into the direct sum of the trivial bundle spanned by $\Omega$ and its orthogonal complement $\overline{K}$, where $K$ is a complex line bundle. Since $\abs{\Omega} = 1$, its covariant derivative is a section of $T^{\ast}X \otimes_{\R} \overline{K}$. Using the almost-complex structure, we get the isomorphism
\eqst{
T^{\ast}X \otimes_{\R} \overline{K} \cong T^{\ast}X \otimes_{\C} K \oplus T^{\ast}X \otimes_{\C} \overline{K}.
}
Moreover, the wedge product gives a complex, bi-linear map 
\eqst{
T^{\ast}X \times T^{\ast}X \longrightarrow \Lambda^2 T^{\ast}X = K.
}
using which, we can identify $TX \cong T^{\ast}X \otimes_{\C} \overline{K}$. 
Thus $\nabla \Omega$ has two components: the first component in $T^{\ast}X \otimes_{\C} K$ is the Nijenhuis tensor and the second one in $TX$ is $d\Omega$. Let $\langle \cdot, \cdot \rangle$ denote the obvious $\overline{K}$-valued pairing between $TX$ and $T^{\ast}X \otimes \overline{K}$.

Let $G = {\rm U}(1)$ and $M = \swann$ be 4-dimensional hyperK\"ahler manifold, which is total space of a Swann bundle, equipped with a tri-Hamiltonian action of ${\rm U}(1)$ that commutes with the permuting ${\rm Sp}(1)$-action. We will call such an action a permuting action of ${\rm U}(2) \cong {\rm Sp}(1) \times_{\pm} {\rm U}(1)$. 

\begin{thm}
\label{thm: main thm 2}
Fix a metric $g_{\sst X}$ on $X$ and let $[g_{\sst X}]$ be its conformal class. Assume that $M$ is obtained as a quotient of a flat, quaternionic space and equipped with a residual permuting action of ${\rm U}(2)$ from the flat space. Then, there exists a 1-1 correspondence between the following:
\begin{itemize}
\item pairs consisting of a metric $g'_{\sst X} \in [g_{\sst X}]$ and a solution $(u, \ms{A})$ to the generalised Seiberg-Witten equations, such that the image of $u$ does not contain a fixed point of the ${\rm U}(1)$ action on $M$
\item pairs consisting of a metric $g''_{\sst X} \in [g_{\sst X}]$ and a self-dual 2-form $\Omega$ satisfying
\eq{
\label{eq: main thm 2}
(\nabla^{\ast}\nabla \Omega)^{\perp} + 2\pair{d\Omega, N_{\Omega}} = 0,~~~  
\frac{3}{2}\abs{N_{\Omega}}^2 + \frac{1}{2}\abs{d\Omega}^2 + \frac{1}{2}\, s_{\sst X}(g''_{\sst X}) < 0 
}
where $s_{\sst X}(g''_{\sst X})$ denotes the scalar curvature with respect to the metric $g''_{\sst X}$.
\end{itemize}
\end{thm}
Theorem \ref{thm: main thm 2} was proved by Donaldson \cite{donaldson} for the usual Seiberg-Witten equations.
\smallskip

Notice that first equation in the second bullet of Theorem \ref{thm: main thm 2} is nothing but a perturbation of Euler-Lagrange equation of the energy functional
\eq{
\label{eq: energy functional twistor section}
\int_{X} \abs{\nabla\Omega}^2.
}
The functional was studied by Wood \cite{wood}. Critical points of the functional correspond to a choice of ``optimal" almost-complex structures, amongst all possible almost-complex structures on $X$. 

\section{Acknowledgements}
The author wishes to thank Prof. Clifford Taubes for pointing out a crucial error in the earlier version of the article. 
A major part of this article is based on the author's doctoral dissertation, during which he was financially supported by the DFG. The author thanks his supervisor Prof. V. Pidstrygach for his unwavering support and constant encouragement.
The author would also like to thank the anonymous referee for many helpful remarks and suggestions, especially on the second half of the article.

\section{Preliminaries and definitions}
\label{preliminaries}

\subsection{HyperK\"ahler manifolds}

A $4n$-dimensional Riemannian manifold $(M, g_{\sst M})$ is \emph{hyperK\"ahler} if it admits a triple of almost-complex structures $I_i \in \End(TM) ~ i=1,2,3$ , which are covariantly constant with respect to the Levi-Civita connection and satisfy quaternionic relations $I_i I_j = \delta_{ijk} I_k$. 

Let ${\rm Sp}(1)$ denote the group of unit quaternions and $\mf{sp}(1)$ denote its Lie algebra. The quaternionic structure on $M$ induces a covariantly constant endomorphism of $TM$ with values in $\mf{sp}(1)^{\ast} = \left(\mf{Im}(\HH)\right)^{\ast}$.
 \eq{
 \label{eq:algebra homomorphism}
 I \in \Gamma(M, \End(TM) \otimes \mf{sp}(1)^{\ast}), ~~~ I_\xi := \xi_1 I_1 + \xi_2 I_2 + \xi_3 I_3, ~~ \xi \in \mf{sp}(1).
 }
Observe that for every $\xi \in S^2 \subset \mf{Im}(\HH)$, the endomorphism $I_{\xi}$ is a complex structure. In other words, $M$ has an entire family of K\"ahler structures parametrized by $S^2$. Define the 2-form
\eqst{
\omega \in \Lambda^2 M \otimes \mf{sp}(1)^{\ast}, ~~ \omega_{\xi}(\cdot, \cdot) = g_{\sst M}(I_{\xi}(\cdot), \cdot).
}
If $\xi \in S^2$, then $\omega_{\xi}$ is just the K\"ahler 2-form associated to $I_{\xi}$.
\begin{defn}
\label{def: permuting action}
An isometric action of ${\rm Sp}(1)$ on $M$ is said to be \emph{permuting} if the induced action on the 2-sphere of complex structures is the standard action of ${\rm SO}(3) = {\rm Sp}(1)/\pm 1$ on $S^2$:
\eqst{ 
dq ~ I_{\xi} ~ dq^{-1} = I_{ q\xi \bar{q}}, ~~ \text{for} ~~ q \in {\rm Sp}(1),~~ \xi \in \mf{sp}(1), ~~ \norm{\xi}^2 = 1 
 }
\end{defn} 

\begin{defn}
\label{def: tri-holomorphic action}
An isometric action of a Lie group $G$ on $M$ is \emph{tri-holomorphic} or \emph{hyperK\"ahler}, if it preserves the hyperK\"ahler structure
\eqst{
\eta_{\ast} I_{i} = I_{i} \eta_{\ast} ~~ i=1,2,3, ~~ \eta \in G.
}
In particular, $G$ fixes the 2-sphere of complex structures on $M$. The action is \emph{tri-Hamiltonian} (or hyperHamiltonian) if it is Hamiltonian with respect to each $\omega_i$. The three moment maps can be combined together to define a single, $G$-equivariant map \emph{hyperK\"ahler moment map}  $\mu: M \longrightarrow \mf{sp}(1)^{\ast}\otimes \mf{g}^{\ast}$, which satisfies 
\eqst{
d(\langle \mu, \xi_{i}\otimes\eta) = \iota_{\sst K^M_{\eta}}\,\omega_i, ~~~ \eta \in \mf{g}, ~~ \xi_i \in \mf{sp}(1) ~~\text{is the basis}
}
and $K^M_{\eta}$ denotes the fundamental vector-field due to the infinitesimal action of $\eta$.
\end{defn}

\begin{defn}
\label{def: hyperkahler potential}
A \emph{hyperK\"ahler potential} is a smooth function $f: M \longrightarrow \R^+$ which is simultaneously a K\"ahler potential for all the three complex structures $I_1, I_2, I_3$.
\end{defn}

\subsection{Target hyperK\"ahler manifold}
Suppose that $M$ is a hyperK\"ahler manifold with a permuting action of ${\rm Sp}(1)$ and a tri-Hamiltonian action of a compact Lie group $G$ which commutes with the ${\rm Sp}(1)$-action. Let $\varepsilon \in G$ be a central element of order two. Let $\Z/2\Z \subset {\rm Sp}(1) \times G$ denote the normal subgroup of order two, generated by the element $(-1, \varepsilon)$. Assume that $\Z/2\Z$ acts trivially on $M$ so that the action of ${\rm Sp}(1) \times G$ descends to an action of ${\rm Spin}^G(3) \, := \, {\rm Sp}(1)\times_{\Z/2\Z} G$. We will refer to this action as a permuting action of ${\rm Spin}^G(3)$. An action of ${\rm Spin}^G(4):= ({\rm Sp}(1)_+ \times {\rm Sp}(1)_-) \times_{\Z/2\Z} G$ is said to be permuting if the action is induced by a permuting action of ${\rm Sp}(1) \cong {\rm Spin}(3)$ via the homomorphism 
\eqst{
\rho: {\rm Spin}^G(4) \longrightarrow {\rm Spin}^G(4)/{\rm Sp}(1)_- \cong {\rm Spin}^G(3).
}
Note that  ${\rm Sp}(1)_-$ acts trivially on $M$.

\subsection{\texorpdfstring{${\rm Spin}^G(4)-$}~ structure}

From the definition of the group ${\rm Spin}^G(4)$, we have the following exact sequence
\eq{
\label{eq: short exact seq sping str}
0 \longrightarrow \Z/2\Z \longrightarrow {\rm Spin}^G(4)\overset{\gamma}{\longrightarrow} {\rm SO}(4) \times \left(G/\{1, \epsilon \} \right) \longrightarrow 0.
}
For simplicity, put $\overline{G} = G/\{1, \epsilon \}$. Let $P_{{\rm SO}(4)}$ denote the frame-bundle of $X$ and $P \rightarrow X$ be a principal $\overline{G}$-bundle over $X$. A ${\rm Spin}^G(4)$-structure over $X$ is a principal ${\rm Spin}^G(4)$-bundle $\pi: Q \rightarrow X$, which is an equivariant double cover of the bundle $P_{{\rm SO}(4)} \times_X P$, with respect to the map $\gamma$ as defined in \eqref{eq: short exact seq sping str}. We refer to \cite{teleman00} for details.

\subsection{Generalised Dirac operator}
\label{subsec:generalised dirac operator}

We define the space of \emph{generalised spinors} to be the space of smooth, equivariant maps 
\eqst{
\Ss \, \, := \, \, C^{\infty}(Q, M)^{{\rm Spin}^G(4)} \, \, \cong \, \, \Gamma(X, Q \times_{{\rm Spin}^G(4)} M).
}
The Levi-Civita connection $\varphi$ on $P_{{\rm SO}(4)}$ and a connection $a$ on the principal $P$ together determine a unique connection on $Q$. Let $\scr{A}$ denote the space of all connections on $Q$, which are the lifts of the Levi-Civita connection. We define the covariant derivative of a spinor $u \in \Ss$, with respect to a connection $\ms{A} \in \scr{A}$ by\footnote{The subscript \emph{hor} implies that $D_{\ms{A}}u$ vanishes on vertical vector fields.}
 \eq{
 \label{covariant derivative}
  D_{\ms{A}} : C^{\infty}(Q,M)^{{\rm Spin}^G(4)} \lra \Hom(TQ, TM)^{{\rm Spin}^G(4)}_{hor} , ~~  D_{\ms{A}} u = du + K^M_{\ms{A}}|_{u}  
  }
where $K^M_{\ms{A}}|_{u}: TQ \rightarrow u^{\ast}TM$ is an equivariant bundle homomorphism defined by $K^M_{\ms{A}}|_{u} (v) = K^M_{\ms{A}(v)}|_{u(p)}$ for $v\in T_p Q$. Denote by $\pi_{\sst {\rm SO}}: Q \longrightarrow P_{{\rm SO}(4)}$ the projection to the frame bundle. Then, alternatively, one can view the covariant derivative as
\al{
\label{eq:covariant derivative alt. def.}
D_{\ms{A}}:~ C^{\infty}(Q,M)^{{\rm Spin}^G(4)} \lra C^{\infty}(Q, (\R^{4})^{*} \otimes TM)^{{\rm Spin}^G(4)}, ~~~ \langle D_{\ms{A}}u(q), w \rangle = du(q)(\widetilde{w})
}
where, $w \in \R^{4}$, $\widetilde{w}$ denotes the horizontal lift of $\pi_{\sst {\rm SO}}(q)(w) \in T_{\pi(q)}X$.

\subsection*{Clifford multiplication}
The second ingredient we need to define the Dirac operator is Clifford multiplication. From \eqref{eq:algebra homomorphism}, we an construct an action of $\mc{C}l^0_4 \cong \mc{C}l_3$ on $TM$ as
\eqst{
\R^3\cong \mf{Im}(\HH) \longrightarrow \End(TM), ~~ h \mapsto I_{h}.
}
The map extends to a ${\rm Spin}^G(4)$-equivariant map $\mc{C}l_3 \longrightarrow \End(TM)$. Thus $TM$ is naturally a $\mc{C}l^0_4 $ module. Now consider $W:= \mc{C}l_4 \otimes_{\mc{C}l^0_4} E$, where $E = (TM, I_1)$. Since $W$ is a $\mc{C}l^0_4$-module, we get a $\Z_2$-graded $\mc{C}l_4$-module
\eqst{
W = W^+ \oplus W^-, ~~ W^+ = \mc{C}l^0_4 \otimes_{\mc{C}l^0_4} E, ~~  W^- = \mc{C}l^1_4 \otimes_{\mc{C}l^0_4} E.
}
 More precisely, $W^+$ is the ${\rm Spin}^G(4)$-equivariant bundle $TM$ with an action induced by $\rho$, whereas $W^-$ is the ${\rm Spin}^G(4)$-equivariant vector bundle $TM$ equipped the left-action: \[ [q_{+}, q_{-}, g] \cdot w_{-} = I_{q_{-}} I_{\bar{q}_{+}} dq_{+} dg ~ w_{-}.\]
 
Identify $\R^{4}$ with $\HH$ by mapping the standard, oriented basis  $(e_{1},e_{2},e_{3},e_{4})$ of $\R^{4}$, to $(1, \bar{i}, \bar{j}, \bar{k})$. The ${\rm Spin}^G(4)$-action on $\HH$ is given by $[q_{+}, q_{-}, g] \cdot h = q_{-}h \bar{q}_{+}$. Clifford multiplication is the ${\rm Spin}^G(4)$-equivariant map
 \eq{
 \label{eq:clifford multiplication}
  \bullet: (\R^{4})^{\ast} \cong \HH \lra \End(W^+ \oplus W^-), ~~
  g_{\R^{4}}(h, \cdot) \longmapsto \begin{bmatrix} 0 & -I_{\bar{h}} \\ I_{h} & 0 
 \end{bmatrix}.
  }
Since $ h \bullet h = -g_{\R^{4}}(h,h)\cdot \id_{W^+ \oplus W^-}$, by universality property, the map $\bullet$ extends to a map of algebras $\bullet: \mc{C}l_{4} \lra \End(W^+ \oplus W^-)$. Composing $\bullet$ with the covariant derivative, we get the \emph{generalised Dirac operator}:
\eq{
\label{eq:gen. dirac operator explicit}
\Dd_{\ms{A}} u \in C^{\infty}(Q, u^{*}W^-)^{{\rm Spin}^G(4)}, ~~~~ \Dd_{\ms{A}}u = \sum_{i=0}^{3} e_{i} \bullet D_{\ms{A}}u(\tilde{e_{i}})
}
where the latter expression follows from equation \eqref{eq:covariant derivative alt. def.}. 

\subsection*{Generalised Seiberg-Witten equations}

Let $\mu$ be a hyperK\"ahler moment map for the $G$-action on $M$ and $a$ be a connection on $P$. Then \emph{generalised Seiberg-Witten equations} for a pair $(u, \ms{A}) \in \Ss \times \scr{A}$, in dimension four, are
\begin{equation}
\label{eq:gen. seiberg-witten}
\left\{\begin{array}{l}
\Dd_{\ms{A}} u = 0 \\
F^{+}_{a} - \Phi(\mu \circ u) = 0
\end{array}\right.
\end{equation}
where $F^+_{a} \in \Map(Q, \Lambda^2_+(\R^4)^{\ast})^{{\rm Spin}^G(4)}$ is the self-dual part of the curvature of $a$ and $\Phi: \mf{sp}(1)^{\ast} \longrightarrow \Lambda^2_+ (\R^4)^{\ast}$ is the isomorphism mapping the basis elements $\xi_l \mapsto \beta_l, \, l = 1,2,3$, where
\eq{
\label{eq: basis of self-dual 2-forms}
\beta_{0} = dx_0 \wedge dx_1 + dx_2 \wedge dx_3, ~ \beta_1 = dx_0 \wedge dx_2 + dx_3 \wedge dx_1, ~ \beta_3 = dx_0 \wedge dx_3 + dx_1 \wedge dx_2.
}
We will supress the isomorphism henceforth.


\section{Conformal transformation of generalised Dirac operator} \label{conformalproperty}

This section is divided into three parts. In the first part, subsection \ref{subsec: metric connections}, we study metric connections for metrics in the conformal class of $g_{\sst X}$. Namely, given the Levi-Civita connection of $g_{\sst X}$ and a metric $g'_{\sst X} \in [g_{\sst X}]$, we explicitly construct the Levi-Civita connection for $g'_{\sst X}$. In the second part, subsection \ref{subsec: example 2 Swann bundles}, we give a quick review of Swann's construction. In the third part, subsection \ref{subsec: gen dirac and conf. metric}, we use the results from subsection \ref{subsec: metric connections} to obtain a formula for conformal transformation of the generalised Dirac operator when the target hyperK\"ahler manifold obtained via Swann's construction. For details on ideas used in this section, we refer the interested reader to \cite{salamon}.

\subsection{Metric connections on conformal bundle}
\label{subsec: metric connections}

\noindent Fix a metric $g_{\sst X}$ on $X$ and let $[g_{\sst X}]$ denote its conformal class. Let $\pi_{1}: P_{{\rm CO}(4)} \longrightarrow X$ denote the bundle of all conformal frames on $(X,[g_{\sst X}])$. A point $p\in P_{{\rm CO }(4)}$ is a ${\rm CO}(4)$-equivariant, linear isomorphism $p: \R^4 \longrightarrow T_{\pi_1(p)}X$. Consider the canonical one-form $\theta: P_{{\rm CO}(4)} \longrightarrow \R^4$ defined as
\eqst{
\theta_{p}(v) = p^{-1}\left((\pi_{1})_{\ast}(v)\right), ~~ p \in P_{{\rm CO}(4)}, ~~v \in T_{p}P_{{\rm CO}(4)}.
}
A metric on $X$ is a section $g_{\sst X} \in \Gamma(X, S^{2}(T^{*}X))$, which can viewed as an equivariant map in $C^{\infty}(P_{{\rm CO}(4)}, S^2(\R^4)^{\ast})^{{\rm CO}(4)}$
\eqst{
\pi_{1}^{\ast}g_{\sst X}\left(\cdot, \cdot \right) = g_{\R^{4}}\left(\theta_{p}(\cdot), \theta_{p}(\cdot)\right) .
}
For a smooth, real-valued function $f$ on $X$, consider the metric $g'_{\sst X}=e^{2(\pi_{1}^{*}f)}g_{\sst X}$ in the conformal class of $g_{\sst X}$. The metrics $g_{\sst X}$ and $g'_{\sst X}$ determine two isomorphic ${\rm SO}(4)$ bundles:
\begin{center}
$P_{{\rm SO}(4)}=\{p\in P_{{\rm CO}(4)}~|~ g_{\R^{4}}(\theta_{p}, \theta_{p}) = \pi_{1}^{\ast}g_{\sst X}(\cdot, \cdot)\}$ \\[0.2cm]  
$P'_{{\rm SO}(4)}=\{p\in P_{{\rm CO}(4)}~|~ g_{\R^{4}}(\theta_{p}, \theta_{p}) = e^{2(\pi_{1}^{\ast}f)}\pi_{1}^{\ast}g_{\sst X}(\cdot, \cdot)\}$
\end{center}
where, $ g_{\R^{4}}(\cdot, \cdot)$ is the standard metric on $\R^{4}$. Let $\varphi$ be a connection on $P_{{\rm CO}(4)}$. Then $\varphi + \theta$ define a 1-form with values in $\mf{co}(4) \oplus \R^4$. We can extend the bracket on the Lie algebra $\mf{co}(4)$ to $\mf{co}(4) \oplus \R^4$ as
\eqst{
[A, x] = - [x, A] = Ax, ~ [x,y] = 0, ~ \text{for} ~ x,y \in \R^4 ~\text{and}~ A \in \mf{co}(4).
}
This defines an affine Lie algebra which is best identified with the frame bundle of $\R^4$.
The failure of the 1-form $\varphi + \theta$ to conform with the associated Maurer-Cartan form is measured by
\eqst{
d(\varphi + \theta) + [\varphi + \theta, \varphi + \theta] = \mc{R}(\varphi) + T(\varphi)
}
where
\eqst{
\mc{R}(\varphi) = d\varphi + \frac{1}{2}[\varphi, \varphi], ~~~ T(\varphi) = d\theta + [\varphi, \theta].
}
Here the entities $\mc{R}$ and $T$ are horizontal 2-forms on the conformal frame bundle, which are nothing but the curvature and the torsion tensors, respectively and the Lie bracket operations are carried out simultaneously with wedging of 1-forms. 

Suppose that $\varphi$ is a connection on $P_{{\rm CO}(4)}$ satisfying
\eq{
(d+\varphi)\, g_{\sst X} = 0 ~~~ \text{and} ~~~ (d+\varphi)\, \theta = 0.
}
Then $ \varphi$ is just the Levi-Civita connection for the metric $ g_{\sst X}$. Let $\varphi'$ denote the Levi-Civita connection for the metric $g'_{\sst X}$. The difference of the 2-connections is a horizontal 1-form on $P_{{\rm CO}(4)}$ and therefore can be written as contraction of $\theta$ with an equivariant function $\xi \in \Hom(\R^{4}, \mf{co}(4)) \cong (\R^{4})^{*} \otimes \mf{co}(4)$. More precisely,
\[ \pair{\theta_p, \,  \xi }(Y) = \pair{\theta_p(Y), \, \xi}, ~~~ Y\in T_{p}P_{{\rm CO}(4)}.\]
Therefore we may write 
\eq{
\label{difference of LC connections}
\varphi'-\varphi \, = \, \pair{\theta, \, \xi} ~~\text{for some}~~ \xi \in (\R^{4})^{*} \otimes \mf{co}(4).
}
Throughout, we will supress the pairing with $\theta$ and simply write $\varphi'-\varphi=\xi$. Consider the covariant derivative of $g'_{\sst X}$ with respect to $\varphi$
\eq{
\label{eq: levi civita phi'}
(d+\varphi)\,(g'_{\sst X}) \, = \, -e^{2(\pi_{1}^{*}f)}\,\, 2\,(\pi_{1}^{*}df)\,g_{\sst X}.
}
The right hand side of the equation can be understood as follows. Define 
\eqst{
f_{i}(p) = \pi_{1}^{*}df \, (\widetilde{p(e_{i})}),
}
where, $e_{i}\in \R^{4}$ is the standard basis element of $\R^4$ and $\widetilde{p(e_{i})}$ is the horizontal lift of $p(e_{i})$ to $P_{{\rm CO }(4)}$ with respect to $\varphi$. We can write 
\eqst{
\pi_{1}^{*}df(p) \, = \, \pair{\sum_{i=1}^4 f_i(p)\, e^i, \, \theta_p}, \,\,\,\,\,\, \sum_{i=1}^4 f_i(p)\, e^i \in (\R^4)^{\ast} \hookrightarrow (\R^4)^{\ast}\otimes \mf{co}(4)
} 
where $e^i$ are the basis for $(\R^4)^{\ast}$. So the action of $\pi_{1}^{*}df$ is just the (left) action of $\sum_{i=1}^4 f_i\, e^i \in \End(\R^4)$.
\begin{rmk}
The negative sign in the equation \eqref{eq: levi civita phi'} is due to the left action of $\aut(\R^{4}) \curvearrowright S^{2}(\R^{4})^{\ast}$, which is given by \[S^{2}(\R^{4})^{\ast} \ni g_{\sst X} \longmapsto b \cdot g_{\sst X}(\cdot, \cdot) := g_{\sst X}(b^{-1}, b^{-1}),\] where $b \in \aut(\R^{4})$.
\end{rmk}

It follows that $\varphi \, + \, \pi_1^{\ast}df$ is a metric connection for $g'_{\sst X}$. But it has a non-zero torsion. Indeed
\eq{
\label{eq: torsion of phi'}
\left(d \, + \, \varphi \,+ \, \pi_1^{\ast}df \right)\,\theta = \pair{\sum_{i=1}^4 f_i\, e^i, \, \theta} \wedge \theta.
}
Point-wise, the torsion tensor is a map 
\eqst{
T(\varphi)(p):\Lambda^{2}D_{p} \cong \Lambda^{2}\R^{4}\xrightarrow{d\theta} \R^{4}.
} 
For the connections $\varphi$ and $\varphi'$ on $P_{{\rm CO}(4)}$, the difference between their torsion tensors is
\eqst{
T(\varphi')_p \, (x\wedge y) - T(\varphi)_p \, (x\wedge y) = \frac{1}{2}(\xi_p(x) \, y -\xi_p(y) \, x), ~~~~ x,y\in\R^{4},
}
In terms of the ${\rm CO}(4)$-equivariant homomorphism:
\eqst{
\delta:(\R^{4})^{*}\otimes \mf{co}(4)\hookrightarrow (\R^{4})^{\ast}\otimes(\R^{4})^{\ast}
\otimes\R^{4}\mapsto\Lambda^{2}(\R^{4})^{*}\otimes\R^{4}\cong \Lambda^{2}(\R^{4})^{*}\otimes(\R^{4})^{\ast}
}
where, the first map is the inclusion and the second one is the anti-symmeterization, we can write $T(\varphi')_p - T(\varphi)_p = -\delta\xi$. Therefore, it follows from \eqref{eq: torsion of phi'} that
\eqst{
\pair{\sum_{i=1}^4 f_i(p)\, e^i, \, \theta} \wedge \theta \, = \, -\delta\left(\sum_{i=1}^4 f_i(p)\, e^i \right).
}
Identify $\mf{so}(4)\cong \Lambda^{2}$ by associating the skew-symmetric endomorphism, to a pair of vectors $v,w\in \R^{n}$, 
\eq{ \label{so4isomorphism}
v\wedge w = \langle v,\cdot \rangle w - \langle w,\cdot \rangle v. 
}
\begin{lem}[\cite{salamon}, Prop. 2.1]
The restriction \[\delta|_{\mf{so}(4)}:(\R^{4})^{\ast}\otimes\Lambda^{2}(\R^{4})^{\ast}\mapsto\Lambda^{2}(\R^{4})^{\ast}\otimes(\R^{4})^{\ast}\]
that maps the difference of two connections to the difference of their torsions is an isomorphism.
\end{lem}
\begin{proof}
Let $a_{ijk} \in (\R^{4})^{\ast}\otimes\Lambda^{2}(\R^{4})^{\ast}$ denote the difference of Christoffel symbols of the two connections. Then, $\delta(a_{ijk}) = \frac{1}{2}(a_{ijk}-a_{jik})$. It is easily seen that if $a_{ijk}\in \text{ker}(\delta)$, then $a_{ijk} = 0$ and hence $\delta|_{\mf{so}(4)}$ is an isomorphism. 
\end{proof}
Suppose that $A$ is the Levi-Civita connection and $B$ is a metric connection on $P_{{\rm CO}(4)}$. Then using the isomorphism $\delta|_{\mf{so}(4)}$, we obtain the expression for $A$ in terms of $B$. Let $B' = B - \alpha$ where $\alpha = \delta|_{\mf{so}(4)}^{-1} (\delta(\xi))$. Then a straightforward computation shows that $T(B') = 0$. This is the strategy we are going to employ to express $\varphi'$ in terms of $\varphi$ and correction terms.

\noindent Pointwise, we can view $\sum_{i=1}^4 f_i\, e^i$ as a 1-form with values in $(\R^4)^{\ast}\otimes\mf{co}(4)$, by writing
\eqst{
\sum_{i=1}^4 f_i\, e^i =  \sum_{i,j} f_{i}\, e^{i} \otimes e^j \otimes e_j \in (\R^4)^{\ast}\otimes (\R^4)^{\ast} \otimes \R^4.
} Using the isomorphism $\R^4 \cong (\R^4)^{\ast}$, we can write the right hand side as
$\sum_{i,j} f_{i}\, e^{i} \otimes e^j \otimes e^j.$ So,
\eqst{
\delta \left(\sum_{i, j} f_i \, e^i \otimes e^j \otimes e^j \right) \, 
 = \, \frac{1}{2} \sum_{i, j} f_i \left( e^i \otimes e^j \otimes e^j -  e^j \otimes e^i \otimes e^j \right)
}
and therefore 
\alst{
 \delta|_{\mf{so}(4)}^{-1} \left[\delta \left(\sum_{i, j} f_i \, e^i \otimes e^j \otimes e^j \right) \right] = \sum_{i, j} f_{i} \, (e^j \otimes e^j \otimes e^i - e^j \otimes e^i \otimes e^j) = -\sum_{i, j} f_{i} \, e^{j}\otimes (e^{i}\wedge e^{j}).
}
It is now easily verified that the torsion
\eqst{
T \left(\varphi + \pi_1^{\ast}df - \delta|_{\mf{so}(4)}^{-1} \left(\delta \left(-\sum_{i=1}^4 f_i\, e^i \right) \right)\right) = -\delta \left(\sum_{i=1}^4 f_i\, e^i \right) - \delta \left(-\sum_{i=1}^4 f_i\, e^i \right) = 0.
}
In conclusion, this is nothing but the Levi-Civita connection for the metric $g'_{\sst X}$ and therefore
\eqst{
\label{diffconnso}
\varphi' = \varphi + \pi_1^{\ast}df + \pair{\sum_{i, j} f_{i} e^{j}\otimes (e^{i}\wedge e^{j}), \theta}.
}
For simplicity, put $\alpha = \pi_1^{\ast}df + \pair{\sum_{i, j} f_{i} e^{j}\otimes (e^{i}\wedge e^{j}), \theta}$.

\begin{prop}[\cite{lawsonmichelsohn} Prop. 6.2, Chap. I]
The adjoint representation induces the Lie algebra isomorphism $\zeta: \mf{spin}(n)\longrightarrow \mf{so}(n)$ is given by:
\[ \zeta(e_{i}e_{j}) = 2e_{i}\wedge e_{j},\] where, $\{e_{i}e_{j}\}_{i<j}$ are the basis elements of $\mf{spin}(n)$. Consequently for $v,w \in \R^{n}$, \[ \zeta^{-1}(v\wedge w) = \frac{1}{4} [v, w] .\]
\end{prop}
Under this isomorphism , $\alpha$ gets mapped to $\sum_{i=1}^4 f_i\, e^i  + \frac{1}{4}\sum_{i, j} f_{i} e^{j}\otimes (e^{i}e^{j} - e^{j}e^{i})$. We denote this again by $\alpha$. 

\subsection{A review of Swann's construction}
\label{subsec: example 2 Swann bundles}

A quaternionic K\"ahler manifold is a $4n$ dimensional manifold whose holonomy is contained in ${\rm {Sp}}(n){\rm {Sp}}(1):= ({\rm {Sp}}(n)\times {\rm {Sp}}(1))/\pm 1$. Let $N$ be a quaternionic K\"ahler manifold of positive scalar curvature and $F$ be the ${\rm {Sp}}(n){\rm {Sp}}(1)$ reduction of the frame bundle $P_{{\rm SO}(4n)}$ of $N$. Then $\mc{S}(N) := F/{\rm {Sp}}(n)$ is a principal ${\rm SO}(3)$-bundle, which is the frame bundle of the three- dimensional vector sub-bundle of skew symmetric endomorphisms of $TN$. The ${\rm {Sp}}(1)$-action, by left multiplication, descends to an isometric action of ${\rm SO}(3)$ on $\HH^{\ast}/\Z_2$. Swann bundle over $N$ is the principal $\HH^{\ast}/\Z_2$ 
\eqst{
\swann := \mc{S}(N)\times_{{\rm SO}(3)}(\HH^{\ast}/Z_2) \longrightarrow N
}

\begin{thm}\cite{swann}
\label{thm: swanns theorem}
The manifold $\swann$ is a hyperK\"ahler manifold with a free, permuting action of ${\rm SO}(3)$ and admits a hyperK\"ahler potential given by $\rho_0 = \frac{1}{2} r^2$. The vector field $\euler = -I_{\xi}K^M_{\xi}$ is independent of $\xi\in \mf{sp}(1)$ and $\grad \rho_0 = \euler$. Moreover, if a Lie group $G$ acts on $N$, preserving the quaternionic K\"ahler structure, then the action can be lifted to a tri-Hamiltonian action of $G$ on $\swann$.
\end{thm}
\noindent The Riemannian metric on the total space $\swann$ is given by $g_{\sst \swann} = g_{\HH^{*}/\Z_{2}} + r^{2}g_{N}$ where $r$ is the radial co-ordinate on $\HH^{*}/\Z_{2}$ and $g_{\HH^{*}/\Z_{2}}$ is the quotient metric obtained from $\HH$. Alternatively, one can write 
\eqst{
\swann = (0,\infty) \times \mc{S}(N)
}
with metric $g_{\mc{U}(N)} = dr^{2} + r^{2}(g_{N} + g_{\mbb{RP}^{3}})$, where $g_{\mbb{RP}^{3}}$ is the quotient metric on $\mbb{RP}^{3}$ derived from its double cover $S^{3}$. Thus, $\swann$ is a metric cone over $\mc{S}(N)$. The manifold $\swann$ is equipped with a natural left action of $\HH^{\ast} \cong \R^+ \times {\rm {Sp}}(1)$
\eq{
\big( (\lambda,q)\, (r,s) \big) \longmapsto (\lambda \cdot r, \, q \cdot s).
}

\subsection{Generalised Dirac operators for conformally related metrics}
\label{subsec: gen dirac and conf. metric}
Henceforth, fix an $M = \swann$, for some quaternionic K\"ahler manifold $N$ of positive scalar curvature and an action of $G$ that preserves the quaternionic K\"ahler structure on $N$. By Theorem \ref{thm: swanns theorem}, the action lifts to a tri-Hamiltonian action of $G$ on $\swann$. Therefore $M$ carries a permuting action of ${\rm Spin}^G(4)$.

Define the conformal ${\rm Spin}^G(4)$ group ${\rm CSpin}^G(4) := \R^{+} \times {\rm Spin}^G(4)$, which is a double cover of ${\rm CO}(4) \times G$ 
\eq{
\label{csping}
0 \longrightarrow \Z/2\Z \longrightarrow {\rm CSpin}^G(4) \overset{\gamma}{\longrightarrow} {\rm CO}(4) \times G \longrightarrow 0.
}
\begin{defn}
A ${\rm CSpin}^G(4)$-structure over $X$ is a principal ${\rm CSpin}^G(4)$-bundle $\widetilde{\pi}: \widetilde{Q} \rightarrow X$, which is an equivariant double cover of bundle $P_{{\rm CO}(4)} \times_X P$, with respect to the map $\gamma$.
\end{defn}
Let $\varphi$ and $\varphi'$ denote the Levi-Civita connections for metrics $g_{\sst X}$ and $g'_{\sst X} \in [g_{\sst X}]$ respectively. Fix a $\overline{G}$-connection $A$ on $P$. Then $A$ uniquely determines the connections $A_{\varphi}$ and $A_{\varphi '}$, which are lifts of $\varphi$ and $\varphi'$ to $\widetilde{Q}$. Then, as shown in subsection \ref{subsec: metric connections},
\eqst{
A_{\varphi '} - A_{\varphi} = \alpha \in C^{\infty}(\widetilde{Q}, \,(\R^4)^{\ast}\otimes \mf{g})^{{\rm Spin}^G(4)}.
}
Consequently, the covariant derivative of $u$, with respect to $A_{\varphi '}$ is
\eq{ \label{diffcovop}
D_{A_{\varphi'}}u = D_{A_{\varphi}}u + K^{M}_{\alpha}|_{u} \in C^{\infty}\big(\widetilde{Q}, (\R^{4})^{\ast}\otimes u^{\ast}TM \big)^{{\rm CSpin}^G(4)}.
}
Recall that $\swann$ admits a hyperk\"ahler potential $\rho_{0}$ and $\euler = \grad \rho_0$. For $\lambda \in \R\setminus \{0\}$,
\eqst{
\rho_{0}(e^{\lambda}x) = \frac{1}{2} g^{M}(\euler|_{e^{\lambda}x}, \euler|_{e^{\lambda}x}) = \frac{1}{2} ~ e^{2\lambda} g^{M}(\euler|_{x}, \euler|_{x}) = e^{2\lambda} \rho_{0}(x).
}
Therefore 
\eqst{
\frac{d}{dt} \rho_{0}(e^{2t \lambda}x)|_{t=0} = d\rho_{0} (\frac{d}{dt}(e^{2t \lambda}x)) = 2d\rho_{0} (K_{\lambda}^{M, \R^{+}}) |_{x} = g^{M}(\euler|_{x}, K_{\lambda}^{M, \R^{+}}|_{x}).
}
On the other hand
\eqst{
\frac{d}{dt} \rho_{0}(e^{2t \lambda}x)|_{t=0} = \frac{d}{dt}(e^{2t\lambda}) \rho_{0}(x) = 2\lambda \rho_{0}(x) = g_{M}(\euler|_{x},\euler|_{x}),
}
which implies that $K_{\lambda}^{M, \R^{+}} = \lambda \euler$.

\noindent We are now in a position to give the proof of Theorem \ref{thm: main thm 1}. But first, we need the following Lemma:

\begin{lem}
\label{gendirac derivative property}
For $f \in C^{\infty}(X, \R)$, we have 
\begin{equation}
\Dd_{A}(e^{-\pi_{1}^{*}f}u) = de^{-\pi_{1}^{*}f}~\Dd_{A}u - \pi_1^{\ast}df \bullet \euler \circ u,
\end{equation}
where $de^{-\pi_{1}^{*}f}$ denotes the differential of the action of $e^{-\pi_{1}^{*}f}$ on $TM$.
\end{lem}
\begin{proof}
Let $p \in \widetilde{Q}$ and $v \in T_{p}\widetilde{Q}$. Let $\gamma:[0,1] \longrightarrow \widetilde{Q}$ be a curve in $\widetilde{Q}$ such that $\gamma(0) = p$ and $\dot{\gamma}(0) = v$. Evaluating the covariant derivative of $e^{-\pi_{1}^{*}f} u$ for $v$:
\eqst{
D_{A}(e^{-\pi_{1}^{*}f} u)(v) = d(e^{-\pi_{1}^{*}f}u)\,(v) + K^{M}_{A(v)}|_{e^{-\pi_{1}^{*}f(p)}u(p)}.
}
The first term of the above expression is
\alst{
d(e^{-\pi_{1}^{*}f} u)(v) 
&= \frac{d}{dt}\big(e^{-\pi_{1}^{*}f} u \big)(\gamma(t))|_{t=0} \\
&= \frac{d}{dt}\big(e^{-\pi_{1}^{*}f(\gamma(t))} u(\gamma(t)) \big)|_{t=0} \\
&= de^{-\pi_{1}^{*}f(p)}du(v) + K^{M}_{\left(-\pi_{1}^{*}df(v)\right)} |_{u(p)}\\
&= de^{-\pi_{1}^{*}f(p)}\, du(v) - \big\langle \sum_{i=1}^4 f_i\, e^i, \, \theta(v) \big\rangle \euler|_{u(p)}
}
and the second term is
\eqst{
K^{M}_{A(v)}|_{e^{-\pi_{1}^{*}f(p)}u(p)} = de^{-\pi_{1}^{*}f(p)} ~ K^{M}_{A(v)}|_{u(p)}.
}
In conclusion,
\eqst{
D_{A}(e^{-\pi_{1}^{*}f} u) = de^{-\pi_{1}^{*}f}~D_{A}u - \big\langle\sum_{i=1}^4 f_i\, e^i, \, \theta \big\rangle \otimes \euler \circ u.
}
Applying Clifford multiplication, proves the statement of the Lemma.
\end{proof}

\begin{proof}[\textbf{Proof of Theorem \ref{thm: main thm 1}}]
With respect to the metric $e^{2\pi^{*}f}g_{\sst X}$, the Clifford multiplication is given by $\bullet' = de^{-\pi_{1}^{*}f} \bullet$.
Substituting for $\alpha$ in \eqref{diffcovop} and applying the Clifford multiplication we get:
\eq{
\label{eq: thm.1 eq 1}
\Dd_{A_{\varphi'}}u 
= de^{-\pi_{1}^{*}f}\left(\Dd_{A_{\varphi}}u + \pi_1^{\ast}df \bullet \euler\circ u + \frac{1}{4} \langle\sum_{i < j} f_{i} e^{j},\, \theta\rangle \bullet K^{M}_{(e^{i}e^{j} - e^{j}e^{i})} |_{u}\right) 
}
Note that in using the identification $(\R^4)^{\ast} \cong \HH$, the element $(e^{i}e^{j} - e^{j}e^{i})$ belongs to the Lie algebra $\mf{sp}(1) \cong \mf{Im}(\HH)$ and has norm 1. Now recall from Theorem \ref{thm: swanns theorem} the vector field $\euler = - I_{\xi}K^M_{\xi}$ is independent of $\xi \in \mf{sp}(1)$. In particular when $\abs{\xi} = 1$, we get $I_{\xi}\euler = K^M_{\xi}$. Therefore, 
\eqst{
K^{M}_{(e^{i}e^{j} - e^{j}e^{i})} |_{u} = I_{(e^{i}e^{j} - e^{j}e^{i})} \euler\circ u = (e^{i}e^{j} - e^{j}e^{i}) \bullet \euler\circ u.
}
Substituting this in \eqref{eq: thm.1 eq 1}, we get
\alst{
\Dd_{A_{\varphi'}}u 
&= de^{-\pi_{1}^{*}f}\left(\Dd_{A_{\varphi}}u + \pi_1^{\ast}df \bullet\euler\circ u + \frac{1}{4} \langle\sum_{i < j} f_{i} e^{j}, \, \theta\rangle \bullet (e^{i}e^{j} - e^{j}e^{i}) \bullet \euler\circ u \right) \\
&= de^{-\pi_{1}^{*}f}\Big(\Dd_{A_{\varphi}}u + \pi_1^{\ast}df \bullet \euler\circ u + \frac{1}{4} \, \langle 4\sum_{i} f_{i}e^{i} - 2\sum_{i,j} f_{i}e^{j}\delta_{i,j} + 4\sum_{i} f_{i}e^{i}, \theta \rangle \bullet \euler\circ u \Big) \\
&= de^{-\pi_{1}^{*}f}\Big(\Dd_{A_{\varphi}}u + \pi_1^{\ast}df \bullet \euler\circ u + \frac{3}{2}\,\pi_{1}^{*}df \bullet \euler\circ u \Big).
}
Now observe that 
\alst{
\Dd_{A_{\varphi'}}(e^{-\pi_{1}^{*}f}u) 
&= de^{-\pi_{1}^{*}f} \left(de^{-\pi_{1}^{*}f} \, \Dd_{A_{\varphi}}u + \frac{3}{2} \, de^{-\pi_{1}^{*}f} \, \pi_{1}^{*}df \bullet \euler\circ u \right) \\ 
&= de^{-\pi_{1}^{*}f}\left(de^{-\frac{5}{2}\pi_{1}^{*}f} \, \Dd_{A_{\varphi}} \, (e^{\frac{3}{2} \pi_{1}^{*}f}u)\right).
}
Thus, in conclusion
\eq{
\label{conformal transformation of gen dirac operator}
\Dd_{A_{\varphi'}} (\scr{B}u) = \scr{B} \left(de^{-5/2\pi_{1}^{*}f} \, \Dd_{A_{\varphi}} \, (e^{3/2\pi_{1}^{*}f}u) \right).
}
\end{proof}

\section{Almost Hermitian geometry and generalised Seiberg-Witten}
In this section, we give the proof of Theorem \ref{thm: main thm 2}. Let the target hyperK\"ahler manifold $M$ be as in Section \ref{subsec: gen dirac and conf. metric}, but with $G= {\rm U}(1)$, so that $M$ now carries a permuting action of ${\rm Spin}^c(4)$. Moreover, let $\text{dim}\,M = 4$. Fix a ${\rm Spin}^c(4)$-structure $Q \rightarrow X$. In this section we restrict our attention to those $\swann$ which can be obtained by a hyperK\"ahler reduction of a flat, quaternionic space. Examples include nilpotent co-adjoint, orbits of complex semi-simple Lie groups, the moduli spaces of instantons on 4-manifolds, etc. We describe this set-up below.

Let $V$ be a finite-dimensional, Hermitian vector space and $\hk:= V \oplus V^{\ast}$. Then $\hk$ is a flat-hyperK\"ahler manifold.  Identifying $\hk$ with $\HH^n$, for some $n$, it is easy to see that $\hk$ carries a natural permuting action of ${\rm Sp}(1)$ given by multiplication by conjugate on the right. Consider the left action of ${\rm U}(1)$ on $\hk$
\eq{
\label{eq: U(1) action on hk}
z \cdot (v, w)  = (z\cdot v, z^{-1} \cdot w).
}
The action is tri-Hamiltonian, with a moment map
\eq{ \label{eq: moment map for S1 action}
\mu_{\R}(v,w) = \frac{1}{2} (\norm{v}^2 - \norm{w}^2), ~~ \mu_{\C}(v,w) = \pair{v,w}
}
Therefore, $\hk$ admits a permuting action of ${\rm U}(2)$. Suppose that another compact Lie group $G \subset U(n) \hookrightarrow Sp(n)$ has a tri-Hamiltonian action on $\hk$ that commutes with the ${\rm U}(2)$-action. Assume zero is a regular value of the $G$-moment map $\mu_{\mf{g}}: \hk \rightarrow \mf{sp}(1)^{\ast}\otimes \mf{g}^{\ast}$. Then, ${\rm U}(2)$ preserves the zero level set of $\mu_{\mf{g}}$ and therefore descends to a permuting action on the quotient $M:= \mu_{\mf{g}}^{-1}(0)/ G$. Put $\widehat{G}:= {\rm Spin}^c(4) \times G$.

\begin{rmk}
More generally, we can consider $\hk = \displaystyle \sum_{i=1}^k V_i \oplus V^{\ast}_i$, where each $V_i$ is a complex representation of ${\rm U}(2) \times G$, equipped with the tri-holomorphic action of ${\rm U}(1)$ by (weighted) left multiplication, so that it may happen that ${\rm U}(1)$ acts non-trivially on the first $\{V_l\}_{l=1}^m, ~ 1<m<k$ and trivially on the rest. However, we require that the image of the spinor be devoid of fixed points of the ${\rm U}(1)$-action. Therefore, we stick to the case where $\hk = V \oplus V^{\ast}$ and ${\rm U}(1) \hookrightarrow {\rm Sp}(n) \curvearrowright\hk$.
\end{rmk}

\subsection{Modified Seiberg-Witten equations}
By assumption $\mu_{\mf{g}}^{-1}(0)/G = M$. Let $P:=\mu^{-1}_{\mf{g}}(0)$ denote the ${\rm Spin}^c$-equivariant principal $G$-bundle over $M$.

\begin{wrapfigure}{r}{0.35\textwidth}
\begin{center}
\vspace{-6mm}
\hspace{-3mm}
\begin{tikzpicture}[->, node distance=2.25cm, auto, shorten >=1pt, every edge/.style={font=\footnotesize, draw}, fill=blue]
         \node (Hn+1)     {$Q$};
         \node (On+1) [right of=Hn+1]    {$M$};
         \node (HPn)  [above of=Hn+1]    {$\widehat{Q}$};
         \node (Xn-1) [right of=HPn]     {$P \subset \hk$};
         \node (X)    [below of=Hn+1]    {$X$};

         \draw        (Hn+1) -- node [left, midway] {$\pi$} (X);        
         \draw        (Hn+1) -- node [above, midway] {$u$} (On+1);
         \draw        (HPn)  -- node [left, midway] {$\pi_{1}$} (Hn+1);
         \draw        (Xn-1) -- node [right, midway] {$\pi_{2}$} (On+1);
         \draw        (HPn)  -- node [above] {$\widehat{u}$} (Xn-1); 
         
\end{tikzpicture}
\vspace{0.5cm}
\end{center}
\end{wrapfigure}
Consider a $\widehat{G}$-bundle $\widehat{Q} \rightarrow X$, as in the diagram. Given a smooth, equivariant map $\widehat{u}: \widehat{Q} \lra \hk$, such that $\mu_{\mf{g}} \circ \widehat{u} = 0$, define $u: Q \rightarrow M$ by $u(q) = \pi_2 (\widehat{u}(p)), ~~ q\in Q, ~~ p \in \pi^{-1}_1(q)$. 
Clearly then, $u$ is a ${\rm Spin}^c(4)$-equivariant map and the diagram commutes. On the other hand, given a smooth spinor $u: Q \lra M$, it defines a principal $\widehat{G}$-bundle over $X$, via pull-back of $P$ and canonically defines $\widehat{u}$, making the diagram commutative. In summary, 
\begin{lem}
There is a bijective correspondence between
\eqst{\{u \in C^{\infty}(Q, \, M)^{{\rm Spin}^c}\} \Longleftrightarrow \{\widehat{u}\in C^{\infty}(\widehat{Q}, \, \hk)^{\widehat{G}}~|~ \mu_{\mf{g}}\circ \widehat{u} = 0\}.
}
\end{lem}
\noindent Fix a connection $\ms{A}$ on $Q$. This is uniquely determined by the Levi-Civita connection on $X$ and a connection $\ms{b}$ on the determinant bundle $P_{{\rm U}(1)}$. The bundle $P \rightarrow M$ is a Riemannian submersion and therefore carries a canonical connection $\ms{a}$. This is defined as follows.
For $p \in P$, let $K^{P, G}_{\eta}|_p$ denote the fundamental vector field at $p$ due to $\eta \in \mf{g}$. For $v \in T_p P$, define $\ms{a}_p(v) \in \mf{g}$ be the unique element such that 
\eqst{
K^{P, G}_{\ms{a}}|_{p}(v) = K^{P, G}_{\ms{a}(v)}|_{p} = -\pr^{\text{im}K^{P, G}}(v)
}
where $\pr^{\text{im}K^{P, G}}$ denotes the orthogonal projection to the vertical sub-bundle, which is nothing but the image of the map 
\eqst{
K^{P, G}: \mf{g} \rightarrow TP, \,\,\,\,\, \eta \longmapsto K^{P, G}(\eta)|_p = K^{P, G}_{\eta}|_p.
}
The pull-back of this connection by $\widehat{u}$, along with the connection $\ms{A}$ on $Q$, uniquely determine a connection $\widehat{\ms{A}}$ on $\widehat{Q}$ (see \cite{victor})
\eq{
\label{eq:unique connection on P_H}
\widehat{\ms{A}} = \pi^{\ast} \ms{A} \oplus \widehat{\ms{A}}_{\mf{g}} \in \Lambda^{1}\left(\widehat{Q}, ~\widehat{\mf{g}}\right)^{\widehat{G}}, ~~ \widehat{\ms{A}}_{\mf{g}} = \widehat{u}^{\ast}\ms{a} - \langle \pi_1^{\ast} \ms{A}, \iota_{\mf{spin}^c} \widehat{u}^{\ast}\ms{a}  \rangle.
}
We can define a twisted Dirac operator $\Dd_{\widehat{\ms{A}}}$ acting on maps $\widehat{u}$.

\begin{prop}
\label{prop:1-1 correspondence}
Then, there is a 1-1 correspondence between
\eq{
\label{eq:1-1 correspondence harmonic spinors}
\{(\widehat{u}, \widehat{\ms{A}})~|~ \Dd_{\widehat{\ms{A}}}\widehat{u} = 0, ~~ \mu_{\mf{g}} \circ \widehat{u} = 0\}~~~\text{and}~~~\{(u, \ms{A})~|~ \Dd_{\ms{A}}u = 0\}.
}
Whenever $\Dd_{\widehat{\ms{A}}}\widehat{u} = 0, ~~ \mu_{\mf{g}} \circ \widehat{u} = 0$ and $\pr_{\mf{g}} \widehat{\ms{A}} = \widehat{\ms{A}}_{\mf{g}}$ as in \eqref{eq:unique connection on P_H} and therefore, $\widehat{\ms{A}}$ is uniquely determined by a ${\rm U}(1)$-connection $a$ on $P_{{\rm U}(1)}$.
\end{prop}
\begin{proof}

For $h\in P$ such that $\mu_{\mf{g}}(h) = 0$, define $\Hh_h := \ker d\mu_{\mf{g}}(h) \cap (\im K^{P, G})^{\perp}$. This is just the horizontal subspace over $h$ with respect to the canonical connection $\ms{a}$. 

We will prove the proposition in two steps. In what follows, we shall denote the $G$ and $Spin^c$-components of $\widehat{\ms{A}}$ by $\widehat{\ms{A}}_{\mf{g}}$ and $\ms{A}$ respectively.

\vspace{0.3cm}
\paragraph{\bf Step 1:}
\label{para:step 1}
In the first step we will prove that $I_{\xi}D_{\widehat{\ms{A}}}\widehat{u}(v) \in \Hh_{\widehat{u}}$ for every $\xi \in \mf{sp}(1)$ and $v \in \Hh_{\widehat{\ms{A}}} \subset T \widehat{Q}$. Indeed, if $\mu_{\mf{g}} \circ \widehat{u} = 0$, then $d\widehat{u}(v) \in \ker d\mu_{\mf{g}}(\widehat{u}(p))$. Also, $K^{P, G}_{\widehat{\ms{A}}_{\mf{g}}}|_{\widehat{u}} \in \ker d\mu_{\mf{g}}(\widehat{u}(p))$ and $K^{P, {\rm Spin}^c}_{\widehat{\ms{A}}}|_{\widehat{u}} \in \ker d\mu_{\mf{g}}(\widehat{u}(p))$. Therefore, $D_{\widehat{\ms{A}}}\widehat{u}(v) \in \ker d\mu_{\mf{g}}(\widehat{u}(p))$. Consequently
\eqst{
0
=\langle d\mu_{\mf{g}} (D_{\widehat{\ms{A}}}\widehat{u}(v)), \xi \otimes \eta\rangle 
= \langle I_{\xi} K^{P, G}_{\eta}|_{\widehat{u}(p)}, D_{\widehat{\ms{A}}}\widehat{u}(v) \rangle
= - \langle K^{P, G}_{\eta}|_{\widehat{u}(p)}, I_{\xi} D_{\widehat{\ms{A}}}\widehat{u}(v) \rangle
}
for $\xi \in \mf{sp}(1), ~ \eta \in {\mf{g}}$ and so $I_{\xi} D_{\widehat{\ms{A}}}\widehat{u}(v) \in (\im K^{P, G})^{\perp}$ for all $\xi \in \mf{sp}(1)$. For $\xi' \in \mf{sp}(1)$,
\eqst{
\langle d\mu_{\mf{g}} (I_{\xi}D_{\widehat{\ms{A}}}\widehat{u}(v)), \xi' \otimes \eta\rangle = \langle d\mu_{\mf{g}} (D_{\widehat{\ms{A}}}\widehat{u}(v)), [\xi, \xi'] \otimes \eta\rangle = 0
}
which implies $I_{\xi} D_{\widehat{\ms{A}}}\widehat{u}(v) \in \ker d\mu_G(\widehat{u}(p))$ for all $\xi \in \mf{sp}(1)$. Thus, $I_{\xi} D_{\widehat{\ms{A}}}\widehat{u}(v) \in \Hh_{\widehat{u}}$.

\vspace{0.3cm}

\paragraph{\bf Step 2:}
\label{para:step 2}
In this step, we prove the equivalence \eqref{eq:1-1 correspondence harmonic spinors}. If $\Dd_{\widehat{\ms{A}}} \widehat{u} = 0$, then from \eqref{eq:gen. dirac operator explicit}, we have
\eqst{
0 = D_{\widehat{\ms{A}}}\widehat{u}(\tilde{e_0}) - \sum^3_{i=1} I_{i}D_{\widehat{\ms{A}}} \widehat{u}(\tilde{e_i})
}

From Step 1, $D_{\widehat{\ms{A}}}\widehat{u}(\tilde{e_0}) \in \Hh_{\widehat{u}}$. It follows that $D_{\widehat{\ms{A}}}\widehat{u}(\tilde{e_i}) \in \Hh_{\widehat{u}}$ for all $i=1,2,3$. Consequently, for any $v \in \Hh_{\widehat{\ms{A}}}, ~~\pr^{\im K^{P, G}}D_{\widehat{\ms{A}}}\widehat{u}(v) = 0$ and we get $\displaystyle K^{P, G}_{\widehat{\ms{A}}_{\mf{g}}(v)} = - \pr^{\im K^{P, G}} d\widehat{u}(v)$. In other words, the ${\mf{g}}$-connection component of $\widehat{\ms{A}}$ is just the pull-back of the canonical connection on $P$. Since the diagram commutes, $d\pi_2(D_{\widehat{\ms{A}}}\widehat{u}) = D_{\ms{A}}u$. Also, as $D_{\widehat{\ms{A}}}\widehat{u}(\tilde{e_i}) \in \Hh_{\widehat{u}}$ for all $i=0, 1,2,3$, we have $\iota^{\ast}I_{i} = \pi^{\ast}_2 \tilde{I_{i}}$ and so,
\eqst{
0
=d\pi_2 (\Dd_{\widehat{\ms{A}}}\widehat{u})
= d\pi_2 \left(D_{\widehat{\ms{A}}}\widehat{u}(\tilde{e_0}) - \sum^3_{i=1} \iota^{\ast}I_{i}~D_{\widehat{\ms{A}}} \widehat{u}(\tilde{e_i})\right)
= \Dd_{\ms{A}}u
}
Thus, $\Dd_{\widehat{\ms{A}}}\widehat{u} = 0$ implies $\Dd_{\ms{A}}u = 0$. On the other hand if $\displaystyle K^{P, G}_{\widehat{\ms{A}}_{{\mf{g}}(v)}} = - \pr^{\im K^{P, G}} d\widehat{u}(v)$ then $D_{\widehat{\ms{A}}}\widehat{u} \in \Hh_{\widehat{u}}$ and so $d\pi_2 (\Dd_{\widehat{\ms{A}}}\widehat{u}) = \Dd_{\ms{A}}u$. Therefore, if $\Dd_{\ms{A}}u = 0$, it implies that $\Dd_{\widehat{\ms{A}}}\widehat{u} \in \im K^{P, G}$. But since, 
\eqst{
\Dd_{\widehat{\ms{A}}}\widehat{u}
= D_{\widehat{\ms{A}}}\widehat{u}(\tilde{e_0}) - \sum^3_{i=1} \pi^{\ast}_2 \tilde{I_{i}}~D_{\widehat{\ms{A}}} \widehat{u}(\tilde{e_i}) \in \Hh_{\widehat{u}}
}
it follows that $\Dd_{\widehat{\ms{A}}}\widehat{u} \in (\im K^{P, G})^{\perp}$ and so $\Dd_{\widehat{\ms{A}}}\widehat{u} = 0$. This proves the statement. 
\end{proof}
With this observation, it is now easy to construct a ``lift" of the equations as follows.

\begin{prop}
\label{prop: 1-1 correspondence of equations}
Fix a connection $a$ on $P_{{\rm U}(1)}$. There is a 1-1 correspondence between the following systems of equations
\eq{\label{eq: modified sw}
\left\{
    \begin{array}{lcl}
      \Dd_{\widehat{\ms{A}}}\widehat{u}= 0 \\      
      F^+_{\ms{b}} - \mu\circ \widehat{u} = 0 \\
      \mu_{\mf{g}} \circ \widehat{u} = 0
    \end{array}
  \right. ~~\text{and}~~~~ \left\{
    \begin{array}{lcl}
     \Dd_{\ms{A}}u= 0 \\      
      F^+_{\ms{b}} - \mu\circ u = 0 
    \end{array}
  \right.}
where $\mu: \hk \rightarrow \rm{i}\R$ denotes the moment map for ${\rm U}(1)$-action on $\hk$. 
\end{prop}
Since the tri-Hamiltonian action of ${\rm U}(1)$ descends to $M$, we denote the ${\rm U}(1)$-moment map by $\mu$ itself. The above correspondence was independently obtained by Pidstrygach \cite{pidstrygach2006} and also by Haydys \cite{haydys2012} (Prop. 4.5 and Thm. 4.6).

\subsection{Almost-complex geometry and generalised Seiberg-Witten}
In this subsection, we give a proof of Theorem \ref{thm: main thm 2}. It exploits the equivalence \eqref{eq: modified sw} and Theorem \ref{thm: main thm 1}. 
Firstly, note that the generalised Seiberg-Witten are not conformally invariant. On the other hand, from Theorem \ref{thm: main thm 1}, we know that the space of harmonic, generalised spinors is conformally invariant. It follows that there is 1-1 correspondence between the solutions $(\widehat{u}', \widehat{\ms{A}}')$ of the system \eqref{eq: modified sw} with respect to the metric $g'_{\sst X} \in [g_{\sst X}]$, such that image of $\widehat{u}$ does not contain a fixed point of the ${\rm U}(1)$-action on $\hk$, and the triples $(g''_{\sst X}, \widehat{u}'', \widehat{\ms{A}}'')$ such that $\abs{\mu\circ \widehat{u}''} = 1$ and $(\widehat{u}'', \widehat{\ms{A}}'')$ satisfy the equations
\eq{
\label{eq: conformally rescaled gen sw}
\left\{
    \begin{array}{lcl}
      \Dd_{\widehat{\ms{A}}''}\widehat{u}'' = 0\\      
      F^+_{\ms{b}} - \lambda\mu\circ \widehat{u}'' = 0 \\
      \mu_{\mf{g}} \circ \widehat{u}'' = 0
    \end{array}
\right.}
where is a strictly positive function given by $\lambda = \abs{\mu\circ u}^{-1}$. To see the correspondence, choose $g''_{\sst X} = \abs{\mu\circ\widehat{u}'}^{-4/3} g'_{\sst X}$. Then $u'' = \abs{\mu\circ\widehat{u}'}^{-1/2} u'$. By virtue of Theorem \ref{thm: main thm 1}, $u''$ is harmonic and the third equation of \eqref{eq: modified sw} remains invariant under the conformal scaling.  Moreover, $\lambda\mu\circ\widehat{u}'' = \mu\circ \widehat{u}'$. The said correspondance follows from the map $(u', \ms{A}') \mapsto (u'', \ms{A}')$.

Suppose we are given a triple $(g''_{\sst X}, \widehat{u}, \widehat{\ms{A}})$ satisfying \eqref{eq: conformally rescaled gen sw} and $\abs{\mu\circ \widehat{u}} = 1$. Then $\Omega = \Phi(\mu \circ \widehat{u})$ is a non-degenerate, self-dual 2-form on $X$, where $\Phi:\mf{sp}(1)^{\ast} \longrightarrow \Lambda^2_+(\R^4)^{\ast}$ is the isomorphism, and defines an almost-complex structure on $X$. 

\begin{lem}
\label{lem: 4d unique cont.}
Suppose that the target hyperK\"ahler manifold $M$ is 4-dimensional. Let $\ms{A}_0$ be a fiducial connection on $Q$ and $u$ be a spinor such that the range  of $u$ does not contain a fixed point of the ${\rm U}(1)$-action on $M$. Then there exists a unique 1-form $\ms{a}_0$ on $X$ such that $\Dd_{\ms{A}}u = 0$, where $\ms{A} = \ms{A}_0 + \imag\ms{a}_0$.
\end{lem}
\begin{proof}
Observe that $\Dd_{\ms{A}}u = \Dd_{\ms{A}_0}u + \sum_{i=0}^3 e^i \bullet K^M_{\imag\ms{a}_0(\widetilde{e_i})}|_u$.  At a point $q \in Q$,
\eqst{
K^M_{\imag\ms{a}_0\,(\widetilde{e_i}(q))} \,|_{u(q)} \, = \, \frac{d}{dt} \exp\left(\,\imag \, t\, \ms{a}_0(\widetilde{e_i}(q))\right) u(q)|_{t=0} \, = \, \left(\ms{a}_0(\widetilde{e_i}(q))\right) K^M_{\imag}|_{u(q)}.
}
Therefore 
\alst{
\Dd_{\ms{A}}u(q)
& = \Dd_{\ms{A}_0}u(q) + \sum_{i=0}^3 \left(\ms{a}_0(\widetilde{e_i}(q)) e^i \right) \bullet K^M_{\imag}|_{u(q)} \\
& = \Dd_{\ms{A}_0}u(q) + \ms{a}_0(q) \bullet K^M_{\imag}|_{u(q)}.
}
Suppose that $\Dd_{\ms{A}}u = 0$. Then, we need to solve the equation
\eqst{
- \Dd_{\ms{A}_0}u = \ms{a}_0 \bullet K^M_{\imag}|_{u}.
}
Point-wise, we can choose identification of $T_{u(q)} M$ and $\R^4$ with quaternions, such that the Clifford multiplication is just the usual quaternionic multiplication. Since the image of $u$ does not contain a fixed point of the ${\rm U}(1)$ action on $M$, $K^M_{\imag}|_{u}$ is a non-vanishing, equivariant section of $u^{\ast}TM \rightarrow Q$. The statement of the Lemma follows.
\end{proof}
In essence, this translates to saying that given a non-vanishing spinor $\widehat{u}$ such that $\mu_{\mf{g}}\circ \widehat{u} = 0$, then there exists a unique 1-form $\ms{a}_0$ on $X$ such that $\Dd_{\widehat{\ms{A}}}\widehat{u} = 0$. Therefore, the connection $\widehat{A}$ is entirely determined by $\widehat{u}$ and hence by the almost complex structure $\Omega = \Phi(\mu\circ \widehat{u})$.

Let $B: \hk \times \hk \longrightarrow \mf{sp}(1)$ denote the symmetric (real) bi-linear form associated to the ${\rm U}(1)$-moment map and $\widetilde{B}$ denote the induced map on $(T^{\ast}X \otimes \hk) \times (T^{\ast}X \otimes \hk)$, obtained using contraction furnished by the Riemannian metric on $X$. Then, $\Omega = B(\widehat{u}, \widehat{u})$ and so
\alst{
\nabla^{\ast}\nabla \Omega
&= 2 \left(B(D^{\ast}_{\widehat{\ms{A}}}D_{\widehat{\ms{A}}}\widehat{u}, \, \widehat{u}) - \widetilde{B}(D_{\widehat{\ms{A}}}\widehat{u}, \, D_{\widehat{\ms{A}}}\widehat{u})\right)}

Applying the Weitzenb\"ock formula 
\eq{ \label{eq: weitzenbock for modified dirac}
\Dd^{\ast}_{\widehat{\ms{A}}}\Dd_{\widehat{\ms{A}}}\widehat{u} = D^{\ast}_{\widehat{\ms{A}}}D_{\widehat{\ms{A}}} \widehat{u} + \frac{s_{\sst X}(g''_{\sst X})}{4} \widehat{u} + F^+_{\ms{b}} \bullet \widehat{u} + F^+_{\widehat{\ms{A}}_{\mf{g}}} \bullet \widehat{u}
}
gives
\eqst{\nabla^{\ast}\nabla \Omega
= -\frac{s_{\sst X}(g''_{\sst X})}{2}  \Omega - B(F^+_{\widehat{\ms{A}}_{\mf{g}}} \bullet \widehat{u}, \widehat{u}) - B(F^+_{\ms{b}} \bullet \widehat{u}, \widehat{u}) - 2 \widetilde{B}(D_{\widehat{\ms{A}}}\widehat{u}, D_{\widehat{\ms{A}}}\widehat{u})
}
We claim that the term $B(F^+_{\widehat{\ms{A}}_{\mf{g}}} \bullet \widehat{u}, \widehat{u})$ vanishes. This follows from the following Lemma:
\begin{lem}
Assume that $\mu_{\mf{g}}(h) = 0$ and let $\xi \in \mf{sp}(1)$ and $\eta \in \mf{g}$. Then 
\eqst{
B(\widehat{u}, ~\eta~\widehat{u}~\overline{\xi}) = 0
}
\end{lem}
\begin{proof}
This follows from the fact that the ${\rm U}(1)$-moment map is $G$-invariant. For $\eta \in \mf{g}$, computing $\frac{d}{dt} B \left(u, \, \exp(t\eta)\, u \, \overline{\xi} \right)|_{t=0}$ proves the statement of the Lemma.
\end{proof}
It follows that $B(F^+_{\widehat{\ms{A}}_{\mf{g}}} \bullet \widehat{u}, \widehat{u}) = 0$.
Therefore,
\eq{
\label{eq: laplacian of omega}
\nabla^{\ast}\nabla \Omega
= -\left(\frac{s_{\sst X}(g''_{\sst X})}{2} + \lambda \right) \Omega - 2 \widetilde{B}(D_{\widehat{\ms{A}}}\widehat{u}, D_{\widehat{\ms{A}}}\widehat{u})
}
We are now is position to give the proof of Theorem \ref{thm: main thm 2}. The arguments of the proof are essentially the same as those of Donaldson's \cite{donaldson}. Nonetheless, for the sake of completeness, we present them here once again.

\begin{proof}[\textbf{Proof of Theorem \ref{thm: main thm 2}}]

Observe that since $\abs{\Omega} = 1$, 
\eqst{
0 = \Delta \abs{\Omega} = 2\pair{\nabla^{\ast}\nabla \Omega, \Omega} - 2\abs{\nabla \Omega}^2.
} 
Using \eqref{eq: laplacian of omega}, we get
\eqst{
2\lambda = -s_{\sst X}(g''_{\sst X}) -2\abs{\nabla \Omega}^2 -  2\pair{\widetilde{B}(D_{\widehat{\ms{A}}}\widehat{u}, D_{\widehat{\ms{A}}}\widehat{u}), \Omega}.
}
Therefore, re-arranging, we have
\eq{
\label{eq: eq1 for lemma for thm 2}
\abs{\nabla \Omega}^2 + \frac{1}{2}\, s_{\sst X}(g''_{\sst X}) +  \pair{\widetilde{B}(D_{\widehat{\ms{A}}}\widehat{u}, D_{\widehat{\ms{A}}}\widehat{u}), \Omega} < 0. 
}
Also, from \eqref{eq: laplacian of omega} we have that $(\nabla^{\ast}\nabla \Omega)^{\perp_{\Omega}} + \widetilde{B}(D_{\widehat{\ms{A}}}\widehat{u}, D_{\widehat{\ms{A}}}\widehat{u})^{\perp_{\Omega}} = 0$. Thus comparing with the identities \eqref{eq: main thm 2} of Theorem \ref{thm: main thm 2}, to complete our proof, we merely need to show that 
\eq{
\label{eq: aim}
\widetilde{B}(D_{\widehat{\ms{A}}}\widehat{u}, D_{\widehat{\ms{A}}}\widehat{u})^{\perp_{\Omega}}  = 2\pair{d\Omega, N_{\Omega}}, ~~\pair{\widetilde{B}(D_{\widehat{\ms{A}}}\widehat{u}, D_{\widehat{\ms{A}}}\widehat{u}), \Omega}
= \frac{1}{4}\left(\abs{N_{\Omega}}^2 - \abs{d\Omega}^2\right).
}
The key issue here is to identify the the map $\widetilde{B}$ on kernel of the Clifford multiplication. In order to do this, it suffices to restrict to the standard model when $X = \R^4$ and the connection $\widehat{\ms{A}}$ is trivial. This is because at any point $x \in X$, there exists a trivialisation in which the connection matrix $\widehat{\ms{A}}$ vanishes at the point $x$. 

Since $\widehat{u} \in \ker \mu_{\mf{g}}$, the derivative $D\widehat{u} \in \Hh_{\widehat{u}} \subset \ker d\mu_{\mf{g}}$. At every point $p \in \ker \mu_{\mf{g}}$, the horizontal subspace $\Hh_p$ can be identified with $T_{\pi_2 (p)}M$. Since $M$ is 4-dimensional, $\Hh_p$ is 4-dimensional and so $\Hh_p \cong \HH$. 

Let $(x_0, x_1, x_2, x_3)$ be the standard co-ordinates on $\R^4$. Let $s_1, s_2,\cdots\cdot s_{2n}$ denote the complex basis for the spinors and write $\widehat{u}$ as
\eqst{
\widehat{u}: \R^4 \longrightarrow \hk, \,\,\,\,\, \widehat{u} = \sum^n_{i=1} f_i~s_i + \sum^{2n}_{i=n+1} g_{i-n} ~ s_{i} \,\,\,\,\, \text{where} \,\,\,\,\, f_i, g_i \in C^{\infty}(\R^4, \C).
}
By Step 2 of Proposition \ref{prop:1-1 correspondence}, $D\widehat{u} \in \Hh_{\widehat{u}}$, which means that without loss of generality, at the origin, we can assume that 
\eqst{
(f_i)_{x_j} = (g_i)_{x_j} = 0 \,\,\,\,\, \text{for} \,\,\,\,\, i = 2, 3, \cdots n \,\,\,\,\, \text{and} \,\,\,\,\, j = 0, 1, 2, 3.
}
Consequently, in the decomposition \eqref{eq: aim}, the only contributing terms are the 1-jets of $f_1, g_1$ at the origin. Therefore, without loss of generality, we can assume that at the origin, $f_i, g_i = 0$ for $i = 2, 3, \cdots\cdot n$. Let $f_0 = f_1(0)$ and $g_0 = g_1(0)$. Then, at the origin $u = f_0 \, s_1 + g_0 \, s_2$. Moreover, since $\abs{\Omega} = 1$, $\abs{f_0}^2 + \abs{g_0}^2 = 1$ and 
\eqst{
 B(\widehat{u}, \widehat{u}) = \left(\frac{\abs{f_0}^2 - \abs{g_0}^2}{2}\right) \,\beta_0 \, + \, \text{Re} \pair{f_0,g_0}\, \beta_1 \, + \, \text{Im} \pair{f_0,g_0}\, \beta_2
}
where $\beta_i$ are the basis of self-dual 2-forms on $\R^4$, given as in \eqref{eq: basis of self-dual 2-forms}. The group ${\rm Spin}(4)$ acts on the base $\R^4$ and also transitively on unit positive spinors. In particular, for a suitable choice of an element in ${\rm Spin}(4)$, we may further assume that at the origin, $f_0 = 1$ and $g_0 = 0$. In particular, $\Omega = \frac{1}{2} \beta_0$ at the origin. Thus $\Omega$ defines the standard complex structure $\frac{1}{2}\,\beta_0$ on $\R^4$. This allows us to use the complex co-ordinates
\eqst{
z = x_0 + ix_1, \,\,\,\,\, w = x_2 + ix_3.
}
From the Dirac equation we have
\eq{\label{eq: dirac eq. on R4}
f_{1\,\overline{z}} = g_{1\,w}, \,\,\,\,\, f_{1\,\overline{w}} = -g_{1\,z}.
}
Moreover, since $f_1 = 1$ at the origin, the derivatives of $f_1$ at the origin are purely imaginary. Therefore, at the origin,
\eq{
\label{eq: imaginary derivative}
f_{1\,z} = -\overline{f_{1\,\overline{z}}} \,\,\,\,\, \text{and} \,\,\,\,\, f_{1\,w} = -\overline{f_{1\,\overline{w}}}.
}
Now, the component of $\widetilde{B}(D\widehat{u}, D\widehat{u})$ along $\frac{1}{2}\,\beta_0$ is
\eqst{
\frac{1}{4}\sum^3_{l=0} \abs{\frac{\partial f_1}{\partial x_l}}^2 - \abs{\frac{\partial g_1}{\partial x_l}}^2 = \frac{1}{16}\left(\abs{f_{1\,z}}^2 + \abs{f_{1\,\overline{z}}}^2 + \abs{f_{1\,w}}^2 + \abs{f_{1\,\overline{w}}}^2 - \abs{g_{1\,z}}^2 - \abs{g_{1\,\overline{z}}}^2 - \abs{g_{1\,w}}^2 - \abs{g_{1\,\overline{w}}}^2 \right).
}
Using the identities \eqref{eq: dirac eq. on R4} and \eqref{eq: imaginary derivative}, we get
\eq{\label{eq: component along w}
\left\langle \widetilde{B}(D\widehat{u}, D\widehat{u}), \,\frac{1}{2}\beta_0  \right\rangle = \frac{1}{16} \left(\abs{g_{1\,z}}^2 + \abs{g_{1\,w}}^2 \right) - \frac{1}{16}\left(\abs{g_{1\,\overline{z}}}^2 + \abs{g_{1\,\overline{w}}}^2 \right).
}
The space orthogonal to $\frac{1}{2}\,\beta_0$ is spanned by $\beta_c = d\overline{z}\cdot d\overline{w}$ and therefore the component of $B(D\widehat{u}, D\widehat{u})$ orthogonal to $\frac{1}{2}\,\beta_0$ is
\alst{
\left(B(D\widehat{u}, D\widehat{u})\right)^{\perp_{\beta_0}} 
&= \sum^3_{l=0} \left[ \left(\frac{\partial f_1}{\partial x_l}\right)^{\dagger} ~ \frac{\partial g_1}{\partial x_l} \right]\,\beta_c \\
&= \frac{1}{4}\left(f_{1\,z} ~ \overline{g_{1\,z}} + f_{1\,\overline{z}} ~ \overline{g_{1\,\overline{z}}} + f_{1\,w} ~ \overline{g_{1\,w}} + f_{1\,\overline{w}} ~ \overline{g_{1\,\overline{w}}}\right)\,\beta_c = \frac{1}{4} \left(g_{1\,z}\,\overline{g_{1\,\overline{w}}} + g_{1\,w} \, \overline{g_{1\,\overline{z}}} \right)\,\beta_c
}
where, once again, we have used the identities \eqref{eq: dirac eq. on R4} and \eqref{eq: imaginary derivative} in the penultimately step. Now $\Omega$ is a section of the twistor bundle and therefore its covariant derivative at the origin is given by the derivative of $f_1 \, \overline{g}_1$ which is nothing but the derivative of $g_1$. The holomorphic part $(g_{1\,z}, g_{1\,w})$ corresponds to the Nijenhuis tensor $N_{\Omega}$ whereas the anti-holomorphic component $(g_{1\,\overline{z}}, g_{1\,\overline{w}})$ corresponds to $d\Omega$, due to the vanishing of the rest of the partial derivatives.

Recall that there is a natural $\overline{K}$-valued pairing between $TX$ and $T^{\ast}X \otimes \overline{K}$. Applying this to $d\Omega$ and $N_{\Omega}$, the pairing corresponds to $\left(g_{1\,z}\,\overline{g_{1\,\overline{w}}} + g_{1\,w} \, \overline{g_{1\,\overline{z}}} \right)~\beta_c$. Therefore, 
\al{
\left(B(D\widehat{u}, D\widehat{u})\right)^{\perp_{\Omega_0}}
&= \frac{1}{4} \times 4\pair{d\Omega, N_{\Omega}}
= \pair{d\Omega, N_{\Omega}} \\
\left\langle \widetilde{B}(D\widehat{u}, D\widehat{u}), \frac{1}{2}\Omega_0  \right\rangle 
&= \frac{1}{16} \times 4 \left(\abs{N_{\Omega}}^2 - \abs{d\Omega}^2 \right) = \frac{1}{4} \left(\abs{N_{\Omega}}^2 - \abs{d\Omega}^2 \right)
}

Substituting in equation \eqref{eq: laplacian of omega}, we have
\eq{
\label{eq: laplacian of omega full}
\nabla^{\ast}\nabla \Omega
= -\left(\frac{s_{\sst X}(g''_{\sst X})}{2} + \lambda\right) \Omega + \frac{1}{2} \left(\abs{d\Omega}^2 - \abs{N_{\Omega}}^2 \right) \Omega - 2 \pair{d\Omega, N_{\Omega}}
}
Also, observe that $\abs{\nabla\Omega}^2 = \abs{d\Omega}^2 + \abs{N_{\Omega}}^2$. The statement of the theorem follows from eq. \eqref{eq: laplacian of omega full} and eq. \eqref{eq: eq1 for lemma for thm 2}. 
\end{proof}

\section{Some Remarks}

For the usual Seiberg-Witten equations, Donaldson remarks that for a fixed metric, the Seiberg-Witten equations are in one-to-one correspondance with solutions to the following equations
\begin{multline}
\label{eq: full ODE of Donaldson}
\nabla^{\ast}\nabla \Omega = - \left(\frac{s}{2} + \abs{\Omega}^2\right) \Omega - 2 \langle d\Omega + \ast d\abs{\Omega}, N_{\Omega} \rangle + \frac{1}{2} \left(\frac{|d\Omega|^2}{\abs{\Omega}^2} - |N_{\Omega}|^2 \right) \Omega \\
+ \frac{1}{2}\left(|d\abs{\Omega}|^2 + 2\langle d\abs{\Omega}, \ast d\Omega \rangle\right) \frac{\Omega}{\abs{\Omega}^2}
\end{multline}

Many examples of hyperK\"ahler manifolds with requisite properties can be obtained via hyperK\"ahler reduction of flat space. Using Prop. \ref{prop: 1-1 correspondence of equations} and applying Donaldson's arguments, one can show that the Abelian, generalised Seiberg-Witten equations, for a 4-dimensional target hyperK\"ahler manifold, can be expressed as \eqref{eq: full ODE of Donaldson}.

Note that the specification of an almost-complex structure $I$ compatible with $\Omega$ imposes a topological restrictions on $X$. Namely, in terms of the Euler characteristic $\chi$ and the signature $\tau$ of $X$,
\eqst{
c_1^2(L) \, = \, 2\,\chi \, + \, 3\, \tau
}
where $L$ is the line-bundle associated to the determinant bundle $P_{{\rm U}(1)}$. For the usual Seiberg-Witten equations, this is precisely the condition under which the expected dimension of the moduli space is zero. Therefore Theorem \ref{thm: main thm 2}, in combination with Donaldson's result \cite{donaldson} delivers a potential candidate to get a compact moduli space.

The arguments in the latter half of the article can be extended for target hyperK\"ahler manifolds of higher dimensions, using similar techniques. However, in this case, one obtains a map from the moduli space of generalised Seiberg-Witten to the usual Seiberg-Witten equations, which is not one-to-one and may not even be surjective.


\end{document}